\documentclass{article}
\usepackage{graphicx}
\usepackage{amsmath,amssymb}

\newtheorem{theorem}{Theorem}[section]

\newtheorem{proposition}[theorem]{Proposition}

\newtheorem{lemma}[theorem]{Lemma}
\newtheorem{corollary}[theorem]{Corollary}

\newtheorem{example}[theorem]{Example}

\begin{document}

\title{Restricted Linear Constrained Minimization of quadratic functionals}
\author{Dimitrios Pappas\\Department of Statistics,\\Athens University of
Economics and Business,\\76 Patission Str, 10434, Athens, Greece
\\(dpappas@aueb.gr, pappdimitris@gmail.com)} \maketitle
\begin{abstract}
In this work a linearly constrained minimization of a positive
semidefinite quadratic functional is examined. Our results are
concerning infinite dimensional real Hilbert spaces, with a singular
positive operator related to the functional, and considering as
constraint a singular operator. The difference between the proposed
minimization and previous work on this problem, is that it is
considered for all vectors perpendicular to the kernel of the
related operator or matrix.
\end{abstract}
\section{Introduction}
\label{intro} The quadratic programming problem with equality
constraints is one of the basic problems in optimization, in both
the finite and the infinite dimensional case. The general problem is
to locate from within a given subset of a vector space the
particular vector which minimizes a given functional. In this case,
the subset of vectors is defined  by a set of linear constraint
relations and the functional is quadratic. In a classical book of
Optimization Theory by Luenberger \cite{luen}, various similar
optimization problems are presented, for both finite and infinite
dimensions.\\In the field of applied mathematics, a strong interest
is shown in applications of the generalized inverse of matrices or
operators. Generalized inverses can be used whenever a matrix/
operator is singular, in many fields of both computational and also
theoretical aspects. An application of the Moore-Penrose inverse in
the finite dimensional case, is the minimization of a positive
definite quadratic functional under linear constraints, presented in
Manherz and Hakimi \cite{hak}.\\The problem studied in their work in
the following:$$ minimize f(x) = \langle x,Qx\rangle + \langle
p,x\rangle + a,  x\in S $$  where $S = \{x:Ax = b\}$ and $Q$ is a
positive definite matrix.\\In this work we will extend this result
for positive semidefinite matrices or operators acting on infinite
dimensional real Hilbert spaces. Since in this case the operator
studied is singular, the proposed minimization is attained for the
vectors perpendicular to the kernel of the operator (or matrix) of
the quadratic functional.
\section{Preliminaries and notation}The notion of the
generalized inverse of a matrix was first introduced by H. Moore in
1920, and again by R. Penrose in 1955. These two definitions are
equivalent and the generalized inverse of an operator or matrix is
also called the Moore- Penrose inverse. It is known that when $T$ is
singular, then its unique generalized inverse $T^{\dagger}$ (known
as the Moore- Penrose inverse) is defined. In the case when $T$ is a
real $r\times m$ matrix, Penrose showed that there is a unique
matrix satisfying the four Penrose equations, called the generalized
inverse of $T$, noted by $T^\dagger$. \noindent The generalized
inverse, known as Moore-Penrose inverse, of an bounded linear
operator $T$ with closed range, is the unique operator satisfying
the following four conditions:
\begin{equation}\label{eq-MooPenr}
   TT^\dagger=(TT^\dagger)^*,\qquad T^\dagger T=(T^\dagger T)^*,\qquad
   TT^\dagger T=T,\qquad T^\dagger TT^\dagger =T^\dagger
\end{equation}
where $T^*$ denotes the adjoint operator of $T$.
\\In what follows, we consider
$\mathcal{H}$ a separable infinite dimensional Hilbert space, $
B(\mathcal{H})$ denotes the set of all bounded operators on
$\mathcal{H} $ and all operators mentioned are supposed to have
closed range. In addition, $\mathcal{R}(T)$ will denote the range of
an operator $T$, and $\mathcal{N}(T)$ will denote its kernel.\\It is
easy to see that $\mathcal{R}(T^\dagger)=\mathcal{N}(T)^\bot$,
$TT^\dagger$ is the orthogonal projection of $\mathcal{H}$ onto
$\mathcal{R}(T)$, denoted by $P_{T}$, and that $T^\dagger T$ is the
orthogonal projection of $\mathcal{H}$ onto $\mathcal{N}(T)^\bot =
\mathcal{R}(T^*)$ noted by $P_{T^*}$. It is well known that
$\mathcal{R}(T^\dagger)=\mathcal{R}(T^*)$.\\It is also known that
$T^\dagger$ is bounded if and only if $T$ has a closed range.\\If
$T$ has a closed range and  commutes with $T^\dagger$, then $T$ is
called an EP operator. EP operators constitute a wide class of
operators which includes the self adjoint operators, the normal
operators and the invertible operators.
\\Let us consider the equation $Tx=b,T \in B(\mathcal{H})$, where $T$ is singular. If $b\notin R(T)$, then
the equation has no solution. Therefore, instead of trying to solve
the equation $\|Tx-b\|=0$, we may look for a vector $u$ that
minimizes the norm $\|Tx-b\|$. Note that the vector $u$ is unique.
 In this case we consider the equation $Tx=P_{R(T)}b$, where
$P_{R(T)}$ is the orthogonal projection on $\mathcal{R}(T)$.
\\The following two propositions can be found in \cite{groe} and hold for operators and matrices:

\begin{proposition}\label{p1}
Let $T \in \mathbb{R}^{r\times m}$ and $b\in \mathbb{R}^{r}, b\notin
R(T)$. Then, for $u\in \mathbb{R}^{m}$, the following are
equivalent:
\begin{enumerate}
\item [(i)]$Tu=P_{R(T)}b$
\item [(ii)]$\|Tu-b\| \leq\parallel Tx-b\|, \forall x \in \mathbb{R}^{m}$
\item [(iii)]$T^*Tu=T^*b$

\end{enumerate}
\end{proposition}

Let $\mathbb{B}=\{u\in \mathbb{R}^{m}| T^*Tu=T^*b\}$. This set of
solutions is closed and convex, therefore, it has a unique vector
with minimal norm. In the literature (eg. Groetsch \cite{groe}),
$\mathbb{B}$ is known as the set of the generalized solutions.

\begin{proposition}\label{p2}
Let $T \in \mathbb{C}^{r\times m}$ and $b\in \mathbb{C}^{r}, b\notin
R(T)$, and the equation $Tx=b$. Then, if $T^\dag$ is the generalized
inverse of $T$, we have that $T^\dag b = u$, where $u$ is the
minimal norm solution defined above.
\end{proposition}
This property has an application in the problem of minimizing a
symmetric positive definite quadratic functional subject to linear
constraints, assumed consistent.\\As mentioned above, EP operators
include normal and self adjoint operators, therefore the operator
$T$ in the quadratic form studied in this work is EP. An operator
$T$ with closed range is called EP if
$\mathcal{N}(T)=\mathcal{N}(T^*)$. It is easy to see that
\begin{equation}
\label{def} T \text{
EP}\Leftrightarrow\mathcal{R}(T)=\mathcal{R}(T^*)\Leftrightarrow\displaystyle{\mathcal{R}(T)\mathop{\oplus}^{\perp}\mathcal{N}(T)=\mathcal{H}}\Leftrightarrow
TT^{\dag}=T^{\dag}T.
\end{equation}We take advantage of the fact that EP operators
have a simple canonical form $T=U_1(A_1\oplus 0)U_1^*$ according to
the decomposition $\mathcal{H}= \mathcal{R}(T)\oplus
\mathcal{N}(T)$. Indeed an EP operator $T$ has the following simple
matrix form, $T = \left[ {\begin{array}{*{20}c}
   A & 0  \\
   0 & 0  \\
\end{array}} \right]$,
where the operator $A:\mathcal{R}(T)\rightarrow \mathcal{R}(T)$ is
invertible, and its generalized inverse $T^\dagger$ has the form
$T^\dagger= \left[ {\begin{array}{*{20}c}
   A^{-1} & 0  \\
   0 & 0  \\
\end{array}} \right]$ (see \cite{campb}, \cite{Driv} ).\\In the finite dimensional case, if  $T\in\mathbb{R}^{n\times n}$ and $\text
{rank}(T)=r$,  the unitary matrix $U$ has the form $[U_1, U_2]$
where $U_1\in \mathbb{R}^{n\times r} $ is an orthonormal basis for
the column space of $T$  and $U_2$ is an orthonormal basis for the
null space of the matrix $T$. In Matzakos and Pappas \cite{pap} an
algorithm is presented for the symbolic computation of the
Moore-Penrose inverse for EP matrices and the corresponding
factorization $T=U(A_1\oplus 0)U^*$ is presented.\\As mentioned
above, a necessary condition for the existance of a bounded
generalized inverse is that the operator has closed range.
Nevertheless, the range of the product of two operators with closed
range is not always closed. In Izumino \cite{Izum} an equivalent
condition is given, using orthogonal projections :
\begin{proposition} \label{pr1} Let $A$ and $B$ be operators with closed range. Then, $AB$ has closed
range if and only if $A^\dagger ABB^\dagger$ has closed range.
\end{proposition}
 We will use the above theorem to prove the existence
of the Moore- Penrose inverse used in our work.
\section{The Generalized inverse in linear constrained minimization}
The Generalized inverse of an operator or matrix plays a crucial
role in many optimization problems where minimal norm solutions are
studied. Ben-Israel and Greville (\cite {Israel}  chapter 3), and
Campbell and Meyer (\cite {campb} chapter 3.6), have considered the
constrained least- squares problem
$$minimize \quad ||Ax -b || \qquad under \quad Bx =d
$$ with $ A, B, b, d, x$ all complex.\\ In Chen \cite{chen},
a similar problem is considered  using partial ordering induced by
cones, where the generalized inverse plays a fundamental role.\\The
purpose of this paper is to extend the work done on this subject by
Manherz and Hakimi \cite{hak}. At first we will extend the finite
dimensional results in the case of positive definite operators, and
then we will examine the case when the operator is singular.
\subsection{Positive Definite Quadratic Forms}
Let $Q$ be a positive definite symmetric matrix. The following
theorem can be found in Manherz and Hakimi \cite{hak} :
\begin{theorem}\label{th3}Let Q be positive definite. Consider the equation $Ax = b$ .\\If the set $S = \{x:Ax = b\}$ is not empty, the the
problem : $$ minimize \quad \Phi(x) = \langle x,Qx\rangle + \langle
p,x\rangle + a,  x\in S $$ and p, a arbitrary has the unique
solution
$$x= Q^{-\frac{1}{2}}(AQ^{-\frac{1}{2}})^\dagger (\frac{1}{2}AQ^{-1}p + b) - \frac{1}{2} Q^{-1} p$$
\end{theorem}
A generalization of the above theorem for infinite dimensional
Hilbert spaces, is by replacing $Q$ with an invertible positive
operator $T$. The operator $A$ must be singular, otherwise this
problem is trivial. In the case of operators instead of matrices,
the proof is similar to Manherz and Hakimi \cite{hak}, but the
existence of a bounded Moore- Penrose inverse is not trivial like in
the finite dimensional case.
\begin{lemma} \label{l1} Let $T \in \mathcal{B}(\mathcal{H})$ be an invertible positive
operator with closed range
 and  $A\in \mathcal{B}(\mathcal{H})$
singular with closed range.\\Then, the range of $AT^{-\frac{1}{2}}$
is closed , where $T^{-\frac{1}{2}}$ denotes the operator
$(T^{\frac{1}{2}})^{-1} $.
\begin{proof}: Using proposition \ref
{pr1} we can see that $$A^\dagger A
(T^{-\frac{1}{2}})(T^{-\frac{1}{2}})^\dagger =
P_{A^*}(T^{-\frac{1}{2}}) (T^{\frac{1}{2}}) = P_{A^*}$$ which has
closed range.
\end{proof}
\end{lemma}
Therefore, the expression of the above theorem in infinite
dimensional Hilbert spaces is the following:
\begin{theorem}\label{th3}Let $T\in B(\mathcal{H})$ be positive definite and the equation $Ax = b$ , where $A$ is singular.\\If the set $S = \{x:Ax = b\}$ is not empty, the the
problem : $$ minimize \quad \Phi(x) = \langle x,Tx\rangle + \langle
p,x\rangle + a,  x\in S $$ with $p, a \in \mathcal{H}$ arbitrary has
the unique solution
$$x= T^{-\frac{1}{2}}(AT^{-\frac{1}{2}})^\dagger (\frac{1}{2}AT^{-1}p + b) - \frac{1}{2} T^{-1} p$$
\end{theorem}\begin{proof} : Since the existence of a bounded generalized inverse for the operator $AT^{-\frac{1}{2}}$ is proved is Lemma \ref{l1}, the rest of the proof of the above theorem is similar to the finite
dimensional case presented in \cite{hak} and is omitted.\end{proof}
\subsection{Positive Semidefinite Quadratic Forms}
An interesting case to examine, is when the positive operator $T$ is
singular, that is, $T$ is positive semidefinite. In this case, since
$ \mathcal{N}(T) \neq\emptyset$,  we have that $\langle x,Tx\rangle
= 0$ for all $x \in \mathcal{N}(T)$ and so, when $\mathcal{N}(T)\cap
S \neq\emptyset$ this problem will take the form $$minimize \quad
\Phi(x) = \langle p,x\rangle + a, \quad for \quad x\in S \cap
\mathcal{N}(T)$$ which is usually treated using the simplex method,
when the space is finite dimensional.
\\A different approach in both the finite and infinite dimensional
case would be to look among the vectors $ x \in \mathcal{N}(T)^\perp
= \mathcal{R}(T^*) = \mathcal{R}(T)$ for a minimizing vector for
$\Phi(x) $. In other words, we will look for the minimum of
$\Phi(x)$ under the constraints $Ax = b, x \in
\mathcal{R}(T).$\\Using the fact that $T$ is an $EP$ operator, we
will make use of the following proposition that can be found in
Drivaliaris et al \cite{Driv}:
\begin{proposition}
\label{prop5} Let $T\in\mathcal{B}(\mathcal{H})$ with closed range.
Then the following are equivalent: \\i) $T$ is EP. \\ii) There exist
Hilbert spaces $\mathcal{K}_1$ and $\mathcal{L}_1$,
$U\in\mathcal{B}(\mathcal{K}_1\oplus \mathcal{L}_1,\mathcal{H})$
unitary and $A_1\in\mathcal{B}(\mathcal{K}_1)$ isomorphism such that
$$T=U(A_1\oplus 0)U^*.$$
\end{proposition}
We present a sketch of the proof for (1)$\Rightarrow$(2):
\begin{proof}:
Let $\mathcal{K}_1=\mathcal{R}(T)$, $\mathcal{L}_1=\mathcal{N}(T)$,
$U:\mathcal{K}_1\oplus\mathcal{L}_1\rightarrow\mathcal{H}$ with
$$U(x_1,x_2)=x_1+x_2,$$ for all $x_1\in\mathcal{R}(T)$ and
$x_2\in\mathcal{N}(T)$, and
$A_1=T|_{\mathcal{R}(T)}:\mathcal{R}(T)\rightarrow\mathcal{R}(T).$
Since $T$ is EP,
$\mathcal{R}(T)\mathop{\oplus}^{\perp}\mathcal{N}(T)=\mathcal{H}$
and thus $U$ is unitary. Moreover it is easy to see that
$U^*x=(P_Tx,P_{\mathcal{N}(T)}x),$ for all $x\in \mathcal{H}$. It is
obvious that $A_1$ is an isomorphism. A simple calculation shows
that
$$T=U(A_1\oplus 0)U^*.$$
\end{proof}
It is easy to see that when $T=U(A_1\oplus 0)U^*$ and $T$ is
positive, so is $A_1$, since $\langle x,Tx\rangle = \langle
x_1,A_1x_1\rangle, x_1 \in \mathcal{R}(T)$.
\\ In what follows, $T$
will denote a singular positive operator with a canonical form
$T=U(A_1\oplus 0)U^*$ , $R$ is the unique solution of the equation
$R^2 = A_1$ and $$X^\dagger =U\left[ {\begin{array}{cc}
   (A_1^{-1})^{\frac{1}{2}} & 0  \\
   0 & 0  \\
\end{array}} \right]U ^* =  U\left[ {\begin{array}{cc}
   R^{-1} & 0  \\
   0 & 0  \\
\end{array}} \right]U ^*$$\\
\begin{theorem}\label{th4}Let T be positive semidefinite and $X^2 = T$. Consider the equation $Ax = b$.\\If the set $S = \{x:Ax = b\}$ is not empty, then the
problem : $$ \text {minimize} \quad \Phi(x) = \langle x,Tx\rangle +
\langle p,x\rangle + a,  x\in S\cap \mathcal{N}(T)^\perp$$ and p, a
arbitrary has the unique solution
$$\hat{x}= X^\dagger(AX^\dagger)^\dagger (\frac{1}{2}AT^\dagger p + b) - \frac{1}{2} T^\dagger
p$$ assuming that the operator $P_{A^*}P_T$ has closed range.
\end{theorem}
\begin{proof}: Since we will restrict the minimization for all vectors $x \in
\mathcal{N}(T)^\perp$ we have that $\langle x,p\rangle = \langle
x,p_1 \rangle$, where $p_1 = P_{\mathcal{R}(T)}p$ for all vectors $p
\in \mathcal{H}$ , according to the decomposition $ p = p_1 + p_2
\in
 \mathcal{R}(T)\oplus \mathcal{N}(T). $\\Let $x, a, p \in \mathcal{H}$ with $a$ and
$p$  arbitrary and not a function of $T$. Hence, the vector
$\hat{x}$ that minimizes $\Phi(x)$ also minimizes
$$\Psi(x) = \langle Tx,x\rangle + \langle x,p\rangle
+\frac{1}{4}\langle T^\dagger p,p\rangle = \langle Tx,x\rangle +
\langle x,p_1\rangle +\frac{1}{4}\langle T^\dagger p_1,p_1\rangle$$
We can easily see that
$$
\parallel Xx + \frac{1}{2}X^\dagger p_1 \parallel ^2 = \Psi(x)$$ and
that $$ X = U(R\oplus 0)U^* , \qquad X^\dagger = U(R^{-1}\oplus
0)U^*$$ Set $y = Xx + \frac{1}{2}X^\dagger p_1 $ , which implies
that $y \in \mathcal{R}(X) = \mathcal{R}(T)$.\\We have that $$Xx = y
-\frac{1}{2}X^\dagger p_1  \Leftrightarrow U(R\oplus 0)U^* x = y
-\frac{1}{2}U(R^{-1}\oplus 0)U^* p_1$$\\ Hence, $$x = U(R^{-1}\oplus
0)U^*y -\frac{1}{2}U(R^{-1}\oplus 0)(R^{-1}\oplus 0)U^* p_1$$ and
so, since $y , p_1 \in \mathcal{R}(T)$ $$ x = X^\dagger y
-\frac{1}{2}T^\dagger p_1,\qquad with \qquad  x\in
\mathcal{R}(T^\dagger)=\mathcal{R}(T) = \mathcal{N}(T)^\perp$$ Since
$Ax = b$, we have that $AX^\dagger y = b + \frac{1}{2}AT^\dagger
p_1$ and therefore, the minimal norm solution is $$\hat{y}  =
(AX^\dagger)^\dagger (b + \frac{1}{2}AT^\dagger p_1)$$By
substitution, we have that $$\hat{x}= X^\dagger(AX^\dagger)^\dagger
(\frac{1}{2}AT^\dagger p_1 + b) - \frac{1}{2} T^\dagger p_1$$ and
since $T^\dagger p =T^\dagger p_1$ for all $p\in \mathcal{H}
\Rightarrow \hat{x}=X^\dagger(AX^\dagger)^\dagger
(\frac{1}{2}AT^\dagger p + b) - \frac{1}{2} T^\dagger p $
\\The only thing that needs to be proved is the fact that the
operator $AX^\dagger$ has closed range and so its Moore-Penrose
Inverse is bounded. Since the two operators $A$ and $X^\dagger$ are
arbitrary, one does not expect that the range of their product will
always be closed. From Proposition \ref{pr1}, this is equivalent to
the fact that the operator $P_{A^*}P_T$ has closed range because
$$A^\dagger AX^\dagger( X^\dagger)^\dagger = A^\dagger AX^\dagger X = A^\dagger A U(R^{-1}\oplus 0)(R\oplus 0)U^* = P_{A^*}P_T$$ and the
proof is completed.
\end{proof}
In the sequel, we present an example which clarifies Theorem
\ref{th4}. In addition, the difference between the proposed
minimization ($x \in \mathcal{N}(T)^\perp$) and the minimization for
all $x \in \mathcal{H}$ is clearly indicated.
\begin{example}
Let $\mathcal{H} = \mathbb{R}^3$, and the positive semidefinite
matrix $$Q = \left[ {\begin{array}{ccc}
   26 & 10 & -2  \\
   10 & 8 & 2  \\
   -2 & 2 & 2
\end{array}} \right]$$ We are looking for the minimum of the functional $$f(u) = \langle u,Qu\rangle +
\langle p,u\rangle + a,\quad u\in \mathcal{N}(Q)^\perp \cap S$$ with
$ p = (1, 2, 3)^T, a =(0,0,0)^T $ and the set of constraints S is
defined as $$S= \{( x,y,z) : 3x+ y +z = -1\}.$$\\The set
$\mathcal{N}(Q)^\perp$ has the form $ u = (2x -3y, x, y)^T,  x, y
\in \mathbb{R}$.\\With easy computations, we can see that all
vectors $u \in \mathcal{N}(Q)^\perp$ satisfying the constraint $Au =
b$, where $ A= \left[ {\begin{array}{ccc}
   3&  1 & 1
\end{array}} \right]$ and $b = -1$, have the form
$$ u = (x, \frac{-8x -3}{5}, \frac{-7x -2}{5})^T$$ The matrices $U,
 X^\dagger$ are
$$U = \left[ {\begin{array}{ccc}
  \frac{3}{\sqrt{10}}  &  \frac{1}{\sqrt{35}} &  \frac{1}{\sqrt{14}}  \\
   0  & \frac{\sqrt{35}}{7}   & -\frac{\sqrt{14}}{7}  \\
   -\frac{1}{\sqrt{10}}  &  \frac{3}{\sqrt{35}}  &  \frac{3}{\sqrt{14}}
\end{array}} \right]\qquad  \left[ {\begin{array}{cc}
   A_1^{-1} & 0  \\
   0 & 0  \\
\end{array}} \right]=  \left[
{\begin{array}{ccc}
   0.0667    &     -0.0623    &     0  \\
   -0.063  &  0.1476      &   0  \\
  0  &  0  &  0
\end{array}} \right]$$ $$X^\dagger = U \left[ {\begin{array}{cc}
   (A_1^{-1})^{\frac{1}{2}} & 0  \\
   0 & 0  \\
\end{array}} \right] U^*= \left[ {\begin{array}{ccc}
  0.1908  &  -0.0295 &  -0.0833  \\
   -0.0295  & 0.2644   & 0.1861  \\
    -0.0833  &  0.1861  &  0.1518
\end{array}} \right]$$
\\Using theorem \ref{th4} we can see that the minimizing vector of
$f(u)$ under $$\{Au=b, u \in\mathcal{N}(Q)^\perp\}$$ is $$\hat{u} =
X^\dagger(AX^\dagger)^\dagger (\frac{1}{2}AQ^\dagger p + b) -
\frac{1}{2} Q^\dagger p = (-0.1019, -0.4369 , -0.2573)^T$$\\The
minimum value of $f(u)$ is then equal to 1.4175
\\In Figure 1 we can clearly see that the minimization of the
functional $f(u)$ for all vectors $ u\in \mathcal{N}(Q)^\perp$
having the form $u = (2x -3y, x, y), x, y \in \mathbb{R}$, belonging
to the plane $3x+ y +z= -1$, is attained.\\In Figure 2 we can see
that among all vectors belonging to $ \mathcal{N}(Q)^\perp  $
satisfying $Au = b$, having the form $ u = (x, \frac{-8x -3}{5},
\frac{-7x -2}{5})^T, x \in \mathbb{R}$, the vector $\hat{u} =
(-0.1019, -0.4369 , -0.2573)^T$ found from Theorem \ref{th4}
minimizes the functional $f(u)$ .
\begin{center}
\begin{figure}[h!] \label{f1}
\includegraphics[width=3.8 in,height=2.8 in]{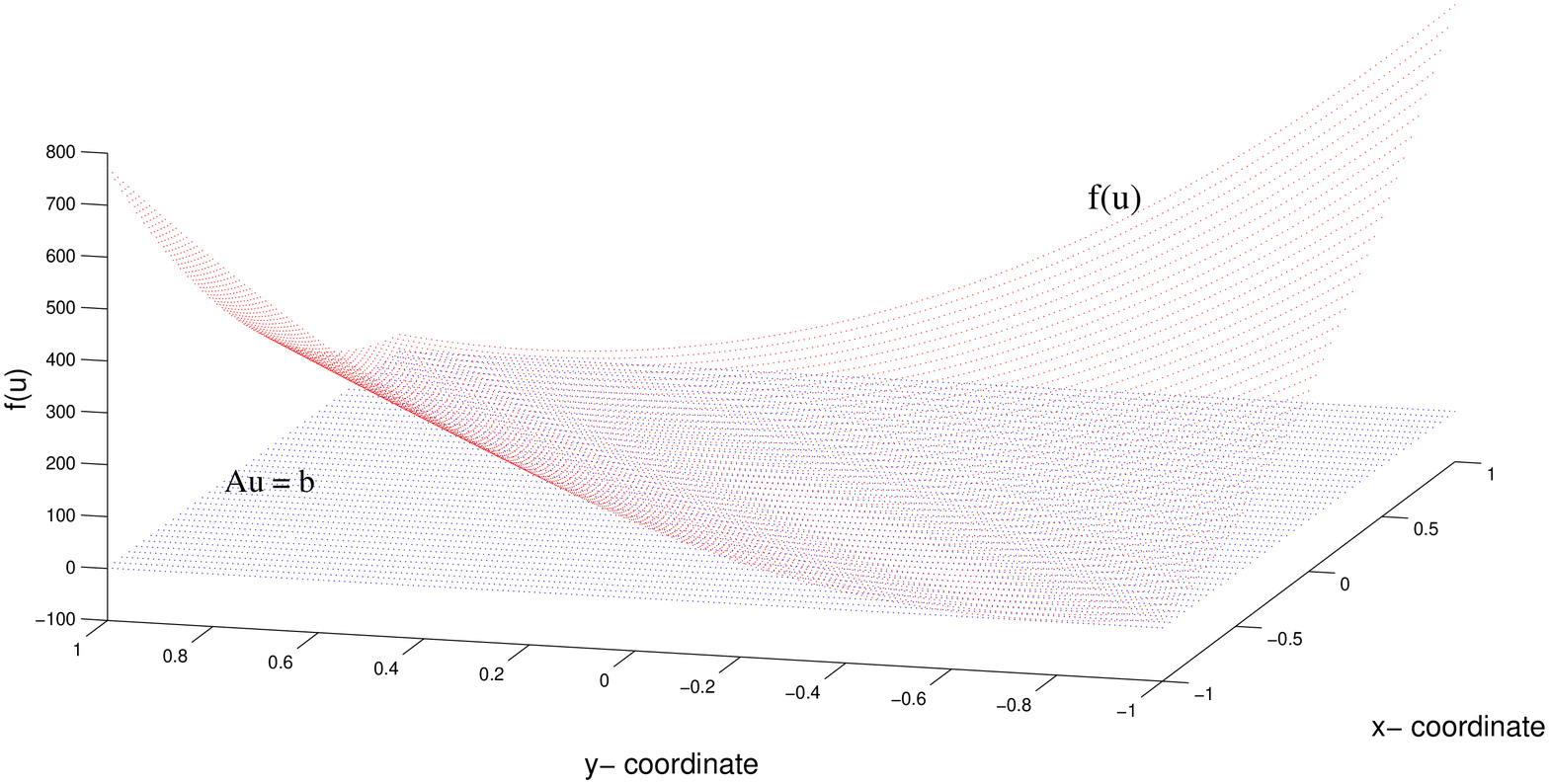}
\caption{Constrained minimization of f(u), $u
\in\mathcal{N}(Q)^\perp $ under Au = b}
\end{figure}
\end{center}
\begin{center}
\begin{figure}[h!] \label{f1}
\includegraphics[width=3.5 in,height=2.8 in]{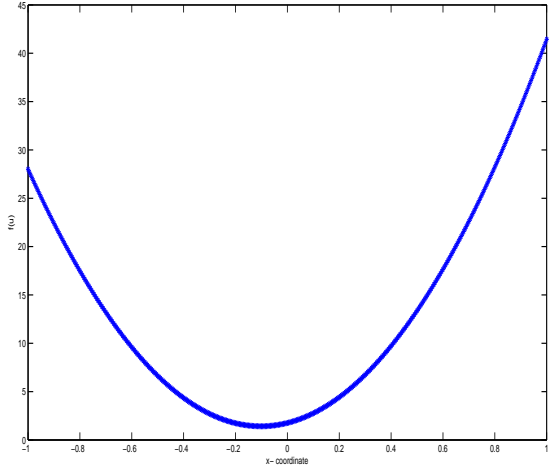}
\caption{f(u), $u \in S\cap \mathcal{N}(Q)^\perp$}
\end{figure}
\end{center}
We will also examine the minimization of $f(u)$ in the cases when $u
\in \mathcal{N}(Q)$ and when $u \in \mathbb{R}^3 $, given that $u\in
S $,  so that the difference between them is clearly indicated:
\begin{itemize}
\item[$(i)$] When the minimization takes place
for all vectors $u \in \mathcal{N}(Q) $, then this problem takes the
form : minimize $ f(u) = \langle p,u\rangle ,  u\in S $.\\We can see
with easy calculations, that the only vector $u \in
\mathcal{N}(Q)\cap S$ is $\tilde{u} = (0.25, -0.5, 0.75)^T$. In this
case , $f(\tilde{u}) = \langle p,\tilde{u}\rangle  = 1.5$
\item[$(ii)$] When the minimization takes place for all vectors $u \in
\mathbb{R}^3 , u\in S $, then the minimizing vector of $f(u)$ is $w
= (-0.1562, 0.3958, -0.9271)^T$ and the minimum value of $f(u)$ is
$f(w) = -1.8229$.
\end{itemize}
\end{example}
\begin{corollary}
We can easily see that, in the case when the vectors $p, a$ are both
equal to zero,  the functional is a constrained quadratic form
$\Phi(x) = \langle Tx,x\rangle$ under $Ax = b$.\\In this case, the
minimizing vector $\hat{u}$ belonging to $\mathcal{N}(T)^\perp  $ is
then equal to $\hat{u}=X^\dagger(AX^\dagger)^\dagger b $ as it was
discussed and proved in \cite{pap2}.
\end{corollary}
\section{Conclusions}
In this work, we propose a constrained minimization in the case of a
quadratic functional related to a positive semidefinite operator.
The proposed minimization takes place for all vectors perpendicular
to the kernel of the corresponding operator. This proposed
constrained minimization method has the advantage of a unique
solution and is easy to implement. Practical importance of this
result in  the finite dimensional case, can be in numerous
applications such as network analysis,  filter design, spectral
analysis, direction finding etc. In many of these cases, the
knowledge of the non zero part of the solution corresponding to the
matrix may be of importance.

\end{document}